\documentclass[11pt,a4paper]{article}
\usepackage{fancyhdr}
\usepackage{amsmath}
\usepackage{amsthm}
\usepackage{amssymb}
\usepackage{wasysym}
\usepackage{paralist}
\usepackage{makeidx}
\usepackage[all]{xy}
\usepackage{mathdots}
\usepackage{yhmath}
\usepackage{mathtools}

\usepackage{graphicx}
\usepackage{epstopdf}

\input xy

\newcommand{\nc}{\newcommand}
 
 \nc{\cl}{\centerline}
 
 \nc{\SL}{{\rm SL}}
 \nc{\hatQ}{{\hat Q}}
 \nc{\sgn}{{\rm sgn}}

 \newcommand{\barS}{{\bar S}}

  \newcommand{\len}{{\rm len}}

  \newcommand{\resp}{{resp.\,}}

\nc\diag{{\rm diag}}
\renewcommand{\vert}{{\,|\,}}
\nc{\hatL}{{\hat L}}

\nc{\barE}{{\bar   E}}
\nc{\D}{{\mathcal D}}
\nc{\E}{{\mathcal E}}
\nc{\F}{{\mathcal F}}
\nc{\even}{{\rm e}}
\nc{\ep}{\epsilon}
\nc{\odd}{{\rm o}}
\nc{\Coker}{{\rm Coker}}
\nc{\olE}{{\overline E}}
\nc{\indBG}{{\rm ind}_B^G\,}
\nc{\indHG}{{\rm ind}_H^G\,}

\nc{\que}{{\mathbb Q}}
\nc{\barlambda}{{\bar\lambda}}
\nc{\barmu}{{\bar\mu}}
\nc{\barm}{{\bar m}}
\nc{\divind}{{\rm div.ind}}
\nc{\tl}{{\tilde{\lambda}}}

\nc{\Sym}{{\rm \Sigma}}

\nc{\Symm}{{\rm Sym}}

\nc{\hatG}{{\hat G}}

\newcommand{\q}{\quad}

\newcommand{\nat}{{\mathbb N}}

\newcommand{\Sp}{{\rm Sp}}

\newcommand{\bs}{\bigskip}

\renewcommand{\vert}{\,|\,}

\renewcommand{\sgn}{{\rm sgn}}

\xyoption{all}

\renewcommand{\vert}{\,|\,}
 \newcommand{\tbw}{\textstyle\bigwedge}
\newcommand{\zed}{{\mathbb Z}}

\newcommand{\GL}{{\rm GL}}

\nc{\geom}{{\rm geom}}
\nc{\rep}{{\rm rep}}

\newcommand{\core}{{\rm core}}

\newtheorem{definition}{Definition}[section]
\newtheorem{proposition}[definition]{Proposition}
\newtheorem{definitions}[definition]{Definitions}
\newtheorem{theorem}[definition]{Theorem}
\newtheorem{lemma}[definition]{Lemma}
\newtheorem{corollary}[definition]{Corollary}

\newtheorem{remark}[definition]{Remark}

\begin{document}


\centerline{\bf Composition Factors of  Tensor Products }
\centerline{\bf of   Symmetric Powers}

\bigskip

\centerline{Stephen Donkin  and Haralampos Geranios}

\bigskip

{\it Department of Mathematics, University of York, York YO10 5DD}\\

\medskip

{\tt stephen.donkin@york.ac.uk,  haralampos.geranios@york.ac.uk}

\bs

\centerline{20  January    2015}
\bs\bs\bs

\section*{Abstract}

\q We determine the composition factors of the  tensor product $S(E)\otimes S(E)$ of two copies of the symmetric algebra of the  natural module $E$ of a general linear group over an algebraically closed field of positive characteristic.  Our  main result  may be regarded as a substantial generalisation of the tensor product theorem of Krop, \cite{Krop}, and Sullivan, \cite{Sullivan} on composition factors of $S(E)$.   We earlier answered the question of which polynomially injective modules are infinitesimally injective in terms of the \lq\lq divisibility index". We are now able to give   an explicit description of the divisibility index for polynomial modules for general linear groups of degree at most $3$.


\bs\bs\bs

\section*{Introduction}

\bs

\q This paper is a continuation of \cite{DG4}.   We are interested in the set of composition factors of the $m$-fold tensor product $S(E)^{\otimes m}$,  of the symmetric algebra $S(E)$ of the natural module $E$ for the general linear group ${\rm GL}_n(K)$ of degree $n$, over an algebraically closed field $K$ of characteristic $p>0$. In \cite{DG4} we related this to the set of composition factors of $\barS(E)^{\otimes m}$, where $\barS(E)$ is the truncated symmetric algebra on $E$. The main result, Theorem 6.5, of \cite{DG4}, is an explicit description of the set of composition factors of $\barS(E)^{\otimes m}$. Here we use this, in the case $m=2$, to give an explicit description of the composition factors of $S(E)\otimes S(E)$.  The description of a composition factor is as a twisted tensor product of \lq\lq primitive" modules and may be regarded as a generalisation of the tensor product theorem of Krop, \cite{Krop}, and Sullivan, \cite{Sullivan}, on the composition factors of $S(E)$.

\q The layout of the paper is the following. In Section 1 we record, in the general context, some properties of bounded, special and good partitions. Section 2 is the technical heart of the paper. In this we determine the $(2,1)$-special partitions. The importance for us is to have control over the composition factors of $L\otimes \tbw^jE$, $j\geq 0$, where $L$ is a composition factor of $\barS(E)\otimes \barS(E)$. The result, Theorem 2.10, is that a partition is $(2,1)$-special if and only if it has the form $\mu+\omega_s$,  for some $s\geq 0$, where    $\mu$ is $2$-special and $\omega_s=1^s$ for $s\geq 1$ and  $\omega_0=0$.  This is key  to  the tensor product description of a composition factor of $S(E)\otimes S(E)$ that we obtain in Section 3. The arguments of Section 2 are highly inductive and somewhat lengthy, involving repeated application of node removal from Young diagrams.   In Section 3 we also give an explicit description, by highest weight, of the polynomial  injective modules for ${\rm GL}_3(K)$ which are injective on restriction to the first infinitesimal subgroup. This is obtained by combining a criterion from \cite{DG2} with our description of the $2$-good partitions, in the special case $n=3$.

\bs

\bs\bs\bs



\section{Generalities on bounded, special and good partitions.}

\q We use the notation and terminology of \cite{DG4}.   Thus  $K$ denotes an algebraically closed field of characteristic $p>0$. We write 
$E$ for the natural module for the general linear group $\GL_n(K)$.   We write $S(E)$ for the symmetric algebra on $E$ and $\barS(E)$ for $S(E)/I$, where $I$ is the ideal generated by $x^p$, $x\in E$. Then $\GL_n(K)$ acts naturally on $S(E)$ and $\barS(E)$ as algebra automorphisms. 
We write $\Lambda^+(n)$ for the set of all partitions of length at most $n$ and for $\lambda\in \Lambda^+(n)$ write $L(\lambda)$ for the irreducible polynomial $\GL_n(K)$-module with highest weight $\lambda$.  For $1\leq m\leq n$ we say that $\lambda\in \Lambda^+(n)$ is $m$-good (\resp $m$-special) if $L(\lambda)$ is a composition factor of $S(E)^{\otimes m}$ (\resp $\barS(E)^{\otimes m}$). 

\q We use the notation $E_n$ and $L_n(\lambda)$ for $E$ and $L(\lambda)$  when we wish to emphasise the role of $n$.

\q We shall need some additional  terminology.

\medskip

\begin{definition} Let $a,b\geq 0$. We shall say that a partition $\lambda$ is $(a,b)$-bounded  if $\lambda_{a+1}\leq b$. 
\end{definition}

Thus a partition $\lambda$ is $(a,b)$-bounded if its diagram fits inside the diagram of a partition of the form $r^ab^s$, for some $r,s\geq 1$.
For  example, $(1,1)$-bounded partitions are hook partitions and   $(2,0)$-bounded  partitions are those with at most two rows. Note that in general a partition $\lambda$ is $(a,b)$-bounded  if and only if the transpose $\lambda'$ is $(b,a)$-bounded.

\begin{definition} We shall say that $\lambda\in \Lambda^+(n)$ is $(a,b)$-special (\resp good), with respect to $n$,    if $L(\lambda)$ is a composition factor of $\barS(E_n)^{\otimes a}\otimes \tbw(E_n)^{\otimes b}$ (\resp $S(E_n)^{\otimes a}\otimes \tbw(E_n)^{\otimes b}$).  Thus $\lambda$ is $m$-special (\resp good) if 
 it is $(m,0)$-special (\resp good), with respect to $n$.
\end{definition}

\begin{remark} Note for any weight $\lambda$ of   $\barS(E_n)^{\otimes a}\otimes \tbw(E_n)^{\otimes b}$ we have $\lambda_1\leq a(p-1)+b$.

\end{remark}

\begin{proposition} Let $\lambda\in\Lambda^+(n)$ be $(a,b)$-good (with respect to $n$). Then $L(\lambda)$ is a composition factor of $L(\mu)\otimes L(\tau)^F$, for some $(a,b)$-special partition $\mu\in \Lambda^+(n)$ and some $a$-good partition $\tau\in \Lambda^+(n)$. In particular   a restricted partition in $\Lambda^+(n)$ is $(a,b)$-good if and only if it is $(a,b)$-special.

\end{proposition}

\begin{proof}
Since $\lambda$ is $(a,b)$-good we have that $L(\lambda)$ is a composition factor of $S(E_n)^{\otimes a}\otimes \tbw(E_n)^{\otimes b}$. Therefore there exist some $a$-good partition $\sigma\in \Lambda^+(n)$ such that $[L(\sigma)\otimes \tbw(E_n)^{\otimes b}:L(\lambda)]\neq0$. Now by \cite{DG4} Proposition 2.8 (applied in the classical case)  we have that since $\sigma$ is $a$-good then $L(\sigma)$ is a composition factor of $L(\xi)\otimes  L(\tau)^F$ for some $a$-special partition $\xi \in \Lambda^+(n)$ and some $a$-good partition $\tau\in \Lambda^+(n)$. Hence $L(\lambda)$ is a composition factor of the tensor product $(L(\xi)\otimes \tbw(E_n)^{\otimes b})\otimes L(\tau)^F$ and more precisely there exist a partition $\mu\in \Lambda^+(n)$ with $L(\mu)$ be a composition factor of $L(\xi)\otimes \tbw(E_n)^{\otimes b}$ such that $[L(\mu)\otimes L(\tau)^F:L(\lambda)]\neq0$. Since $\xi$ is $a$-special it is clear that $\mu$ is $(a,b)$-special and so we are done. The result for the case where $\lambda$ is restricted follows now directly from the above and using the fact that $\tau$ must be $0$ in this case by \cite{DG4}, Lemma 3.1.

\end{proof}

\q We will use freely in the rest of the paper  the fact that for a restricted partition $\lambda\in\Lambda^+(n)$ the notions of $(a,b)$-good and $(a,b)$-special coincide.

\medskip

\begin{lemma} A partition $\mu$ is $(a,b)$-good  if and only if we have $[\nabla(\lambda):L(\mu)] \neq 0$ for some $(a,b)$-bounded partition  $\lambda$.

\end{lemma}

\begin{proof}

This follows in the same way as  \cite{DG4}, Lemma 2.2 using both versions of Pieri's formula, \cite{Mac},  I, Chapter 5, (5.16), (5.17), this time.

\end{proof}

\q We denote by $\core(\lambda)$ the $p$-core of  a partition $\lambda$.  If $\lambda,\mu\in \Lambda^+(n)$ are in the same block (i.e., the polynomial modules $L(\lambda)$ and $L(\mu)$ for ${\rm GL}_n(K)$ lie in the same block) then $\lambda$ and $\mu$ have the same $p$-core, see e.g., \cite{D2}. Thus $[\nabla(\lambda):L(\mu)]\neq 0$ implies that $\core(\lambda)=\core(\mu)$. Using now Lemma 1.5 we have the following.

\begin{corollary}  Let $\lambda\in \Lambda^+(n)$. 

(i) If $\lambda$ is $(a,b)$-good then $\core(\lambda)$ is $(a,b)$-good.

 (ii) Suppose  $\lambda$ is a $p$-core. Then $\lambda$ is $(a,b)$-good if and only if it is $(a,b)$-bounded.

\end{corollary}

 It is easy to check the following stability result by the truncation arguments of \cite{DG4}, Section 3.

\begin{proposition}  Fix $a,b\geq 0$ and let $\lambda$ be an element of $\Lambda^+(n)$.  Then $\lambda$ is $(a,b)$-special (\resp good) if and only if $\lambda$ is $(a,b)$-special (\resp good) when regarded as an element of $\Lambda^+(n')$,  $n'\geq n$.

\end{proposition}

\begin{definition} In view of the proposition above we declare a partition $\lambda$ to be $(a,b)$-special (\resp good) if it is $(a,b)$-special with respect to $n$, for $n$  greater than or equal to the length of $\lambda$.

\end{definition}

\q By the argument of   (cf \cite{DG4}, Section 2,  Lemma 2.4) we have the following result.

\begin{lemma}    If $\lambda$ is $(a,b)$-special (\resp good) and $\mu$ is $(c,d)$-special (\resp good) then $\lambda+\mu$ is $(a+c,b+d)$-special (\resp good).

\end{lemma}

\q On $\barS(E_n)$ and $\tbw(E_n)$ we have the forms  $\barS(E_n)\times \barS(E_n)\to L$ and $\tbw(E_n)\times \tbw(E_n)\to M$, given by multiplication followed by projection onto the top  component, where $L=L_n((p-1)\omega_n)$ and $M=L_n(\omega_n)$.  Thus we have the product form $(\barS(E_n)^{\otimes a}\otimes \tbw(E_n)^{\otimes b})\times (\barS(E_n)^{\otimes a}\otimes \tbw(E_n)^{\otimes b})\to N$, where $N=L_n((a(p-1)+b)\omega_n)$.  For $\lambda\in \Lambda^+(n,r)$, with $\lambda_1\leq a(p-1)+b$, we define $\lambda^\dagger=(a(p-1)+b)\omega_n-w_0\lambda$, where $w_0$ is the longest element of $\Symm(n)$.  The following is obtained as in the proof of Lemma 3.4 of \cite{DG4}.

\begin{proposition} (Reciprocity Principle)    Let $\lambda\in \Lambda^+(n)$ with $\lambda_1\leq a(p-1)+b$. Then  $\lambda$ is $(a,b)$-special if and only if $\lambda^\dagger$ is $(a,b)$-special.

\end{proposition}

\begin{proposition}   Let $n\geq 2$ and $a,b\geq 0$. If $\lambda=(\lambda_1,\ldots,\lambda_n)$ is an $(a,b)$-special (\resp good)  partition  then $(\lambda_1,\ldots,\lambda_{n-1})$ and $(\lambda_2,\ldots,\lambda_n)$ are $(a,b)$-special (\resp good). 

\end{proposition}

\begin{proof}  We give the argument for $(a,b)$-good. The $(a,b)$-special case is similar. We put $\barlambda=(\lambda_1,\ldots,\lambda_{n-1})$. Consider the natural module $E=E_n$ for $\GL_n(K)$. We have $E_n=E_{n-1}\oplus L$, where $L$ is the $K$-span of $e_n$ (and $E_{n-1}$ is the $K$-span of $e_1,\ldots,e_{n-1}$).  We regard $H=\GL_{n-1}\times \GL_1$ as a subgroup of $\GL_n(K)$, in the obvious way.  Then $E_n=E_{n-1}\oplus L$ is an $H$-module decomposition. Since $L(\lambda)$ is a composition factor of $S(E_n)^{\otimes a}\otimes \tbw(E)^{\otimes b}$ it is a composition factor of $S^\alpha E_n\otimes \tbw^\beta E_n$, for some sequences $\alpha=(\alpha_1,\ldots,\alpha_a)$ and $\beta=(\beta_1,\ldots,\beta_b)$ with at most $a$ and $b$ parts.    The $H$-module $L(\lambda)$ has highest weight $\lambda$ and so has ($H$-module) composition factor $L_{n-1}(\barlambda)\otimes L_1(\lambda_n)$.

\q For  $r,s\geq 0$, $S^r(E)=\bigoplus_{r=r_1+r_2} S^{r_1}(E_{n-1}) \otimes S^{r_2} L$ and  \\
$\tbw^s(E_n)=\bigoplus_{s=s_1+s_2} \tbw^{s_1}(E_{n-1})\otimes \tbw^{s_2} L$, as $H$-modules. It follows that \\
$L_{n-1}(\barlambda)\otimes L_1(\lambda_n)$ must be a composition factor of a module of the form $S^{u_1}(E_{n-1})\otimes \cdots\otimes S^{u_a}(E_{n-1})\otimes \tbw^{v_1}(E_{n-1})\otimes\cdots \otimes  \tbw^{v_b}(E_{n-1})\otimes M$, for some $u_1,\ldots,u_a\geq 0$. $v_1,\ldots,v_b\geq 0$ and one dimensional $\GL_1(K)$-module $M$.  Restricting to $\GL_{n-1}(K)$ gives that $\barlambda$ is $(a,b)$-good.

\q The result for $(\lambda_2,\ldots,\lambda_n)$ is obtained by restricting to $\GL_1(K)\times \GL_{n-1}(K)$ and arguing in the same way.

\end{proof}

\begin{corollary} Suppose $\lambda\in \Lambda^+(n)$ and $\lambda_1=a(p-1)+b$. Then $\lambda$ is $(a,b)$-special if and only if $(\lambda_2,\ldots,\lambda_n)$ is $(a,b)$-special. 

\end{corollary}

\begin{proof} Applying the reciprocity principle with respect to $n$ we get that $\lambda$ is $(a,b)$-special if and only if $(a(p-1)+b)\omega_n-(\lambda_n,\ldots,\lambda_1)$ is $(a,b)$-special, i.e., if and only if $(a(p-1)+b)\omega_{n-1}-(\lambda_n,\ldots,\lambda_2)$ is $(a,b)$-special. Now applying the reciprocity principle with respect to $n-1$ we see that this holds if and only if 
$(\lambda_2,\ldots,\lambda_n)$ is $(a,b)$-special.

\end{proof}

\subsection*{Bounded Modules}

\q We remind to the reader that a filtration $0=V_0\leq V_1\leq \cdots\leq V_r=V$ of  a finite dimensional polynomial $GL_n(K)$-module $V$ is said to be good if for each $1\leq i\leq r$ the quotient $V_i/V_{i-1}$ is either zero or isomorphic to $\nabla(\lambda^i)$ for some $\lambda^i\in \Lambda^+(n)$.  For a polynomial $GL_n(K)$-module $V$ admitting a good filtration for each $\lambda\in \Lambda^+(n)$, the multiplicity 
$|\{1\leq i\leq r\vert V_i/V_{i-1}\cong \nabla(\lambda)\}|$ is independent of the choice of the good filtration, and will be denoted $(V:\nabla(\lambda))$. 
 
\medskip

\q We fix $a,b\geq 0$.

\begin{definition}  Let $M$ be a finite dimensional polynomial module with a good filtration. We say that $M$ is $(a,b)$-bounded if each $\lambda\in \Lambda^+(n)$ such that $(M:\nabla(\lambda))\neq 0$ is $(a,b)$-bounded.  We say that $M$ is $(a,b)$-deficient if  $(M:\nabla(\lambda))=0$ for every $(a,b)$-bounded element $\lambda$ of $\Lambda^+(n)$. 

\end{definition}

\begin{remark} Note that if $M$ is a finite dimensional polynomial  module with a good filtration and character $\chi=\sum_{\lambda\in \Lambda^+(n)} r_\lambda \chi(\lambda)$ then $M$ is $(a,b)$-bounded if $\lambda$ is $(a,b)$-bounded whenever $r_\lambda\neq 0$ and $M$ is a $(a,b)$-deficient if $r_\lambda=0$ for all $(a,b)$-bounded $\lambda$.  Here $\chi(\lambda)$ is the character of $\nabla(\lambda)$, i.e., the Schur function corresponding to the partition $\lambda$.
\end{remark}

\begin{lemma}  Let $M$ be finite dimensional polynomial module with a good filtration and suppose that $M$ is $(a,b)$-deficient. Then for every finite dimensional polynomial module $V$ with a good filtration the polynomial module $M\otimes V$ is $(a,b)$-deficient.

\end{lemma}

\begin{proof} By the above remark it is enough to show that the coefficient of $\chi(\lambda)$ in the character of $M\otimes V$ is zero for all $(a,b)$-bounded 
$\lambda\in \Lambda^+(n)$.  It follows that it is enough to note that for $\lambda,\mu\in \Lambda^+(n)$ with $\lambda$  being not $(a,b)$-bounded
 the coefficient of $\chi(\tau)$  in $\chi(\lambda)\chi(\mu)$ is $0$ for all $(a,b)$-bounded $\tau\in \Lambda^+(n)$.  So it is enough to show that for any symmetric function $\psi$ in $n$ variables $ \psi\chi(\lambda)$ is a $\zed$-linear combination of Schur symmetric functions $\chi(\tau)$ with $\tau$ not $(a,b)$-bounded. The ring of symmetric function is generated by the elementary symmetric functions $e_r=\chi(1^r)$, for $1\leq r\leq n$ so it enough to show that each $e_r\chi(\lambda)$ is a sum of terms $\chi(\tau)$, with $\tau$ not $(a,b)$-bounded. However, by Pieri's formula $e_r\chi(\lambda)$ is a sum of terms $\chi(\tau)$ where the diagram of $\tau$ is obtained by adding boxes to the diagram of $\lambda$, so the result is clear.

\end{proof}

\q As in \cite{DG4}, we write $I(\lambda)$ for the injective indecomposable polynomial module corresponding to the partition $\lambda$. If $\lambda\in \Lambda^+(n)$ then we write $I_n(\lambda)$ for this  polynomial injective ${\rm GL}_n(K)$-module if we wish to emphasise the role of $n$.

\begin{lemma} Let $\lambda\in \Lambda^+(n)$. Then $\lambda$ is $(a,b)$-good if and only if $I(\lambda)$ is not $(a,b)$-deficient.

\end{lemma}

\begin{proof} We have that $\lambda$ is $(a,b)$-good if and only if there exists some $(a,b)$-bounded partition $\mu$ such that $[\nabla(\mu):L(\lambda)]\neq 0$. By reciprocity, see e.g., \cite{DqSchur}, Section 4, (6),  this is if and only if there exists an $(a,b)$-bounded partition $\mu$ such that $(I(\lambda):\nabla(\mu))\neq 0$, i.e.,  if and only if $I(\lambda)$ is not $(a,b)$-deficient.

\end{proof}

\q We   give some additional standard terminology.

\medskip

\begin{definitions}

{\rm We denote the length of a partition $\lambda$ by $\len(\lambda)$. Let $\lambda$ be a partition.  

(i)  We call a node $R$ of $\lambda$ (or more precisely  the diagram of $\lambda$)    {\it removable} if the removal of $R$ from the diagram of $\lambda$ leaves the diagram of a partition, which will be denoted $\lambda_R$. Thus the node  $R$ is removable node if it has the form $(i,\lambda_i)$ for some $1\leq i\leq  \len(\lambda)$ and either $i=\len(\lambda)$ or 
$\lambda_i > \lambda_{i+1}$.

(ii)  An {\it addable} node $A$ of $\lambda$ is an element of $\nat\times \nat$ such that the addition of $A$ to the diagram of $\lambda$ gives the diagram of a partition, which will be denoted $\lambda^A$. Thus $A$ is addable if it has the form $(i,\lambda_i+1)$ for some $1\leq i\leq \len(\lambda)+1$ and either $i=1$ or $\lambda_{i-1}>\lambda_i$ or $A=(\len(\lambda)+1,1)$. 

(iii)  The {\it residue} of a node $A=(i,j)$ of a  $\lambda$ is defined to be the congruence class of  $i-j$ modulo $p$. 

(iv) Let $A$ and $B$ be removable or addable nodes of $\lambda$. We  shall say that  $A$ is {\it lower} than $B$ if  $A=(i,r)$, $B=(j,s)$ and $i>j$. 

(v) We call a  removable node $R$ {\it suitable} if the residue of $R$ is different from the residue of $A$ for all lower addable nodes of $\lambda$.}

\end{definitions}

\medskip

\begin{lemma}   Let $a,b\geq 0$. Suppose that $\lambda$ is a partition and $R=(h,\lambda_h)$ is a suitable node. Then for all sufficiently large $n$ we have:

(i) $I_n(\lambda)$ is a direct summand of $I_n(\lambda_R)\otimes E_n$;

(ii)  if $I_n(\lambda_R)$ is $(a,b)$-deficient then so is $I_n(\lambda)$.

Furthermore if  $\lambda$ is $(a,b)$-good then so is $\lambda_R$.

\end{lemma}

\begin{proof}  See the proof of Lemma 3.11 of \cite{DG4}.

\end{proof}

\bs\bs\bs



\section{The determination of the $(2,1)$-special partitions.}

\q We are interested in the $(2,1)$-special partitions.  We first record the description of $2$-special partitions  obtained by taking $m=2$ in   \cite{DG4}, Theorem 6.4.  Recall that a partition is $1$-special if and only if it has the form $(p-1)^ka$, for some $k\geq 0$ and $p-1>a\geq 0$, see e.g., \cite{DG4}, 4.13 Remark. We split the description into the restricted and non-restricted cases.   In order to cut down the number of special cases considered we assume now that $p>2$. This is harmless since the main result of this section, Theorem 2.10 is easily seen to hold also in the characteristic $2$ case.

\begin{lemma} A restricted partition $\lambda$ is $2$-special  if and only if has one of the following forms:

(i) $\lambda=(p-2)^kab$ with $k\geq0$ and $p-2>a\geq b\geq 0$;

(ii) a sum of two $1$-special partitions, i.e., we have  \\
$\lambda= (2(p-1))^j(p-1+a) (p-1)^kb$, 
with $j,k\geq 0$  and $0\leq a,b\leq p-2$, if $k\geq1$ and $0\leq a\leq b\leq p-2$ if $k=0$.

\end{lemma}

\begin{proposition} Let $\lambda$ be a partition and write $\lambda=\lambda^0+p\barlambda$, with $\lambda^0$ restricted.  Then $\lambda$ is $2$-special and non-restricted  if and only if \\
$\lambda^0=(p-2)^kab$, with $k\geq 0$, $p-2 >  a\geq b\geq 0$  and $\barlambda=\omega_s$, with $s\geq 1$.

\end{proposition}

\q We now embark on the determination of the $(2,1)$-special partitions.    In this section congruent will always mean congruent modulo $p$. Recall that partitions $\lambda$ and $\mu$ of  the same degree  have the same $p$-core if and only if for each $0\leq m<p$ the number of nodes of (the diagram of) $\lambda$ of residue $m$ is equal to the number of nodes of $\mu$ with residue $m$.  

\medskip

\begin{lemma} Let $\lambda$  be a partition with first entry  $\lambda_1\leq p-1$. Then $\lambda$ is $(2,1)$-good if and only if it has the form $(p-2)^kab+\omega_r$,  for some $k\geq 0$, $p-2>a\geq b\geq 0$ and $r\geq 0$. 
\end{lemma}

\begin{proof} Certainly all these partitions are $(2,1)$-good by Lemma 1.9   since partitions of the form $(p-2)^kab$   are $(2,0)$ good and $\omega_r$ is $(0,1)$-good.

\q We claim now that a $(2,1)$-good partition $\lambda$ with $\lambda_1\leq p-1$ has the required form. Certainly the result holds for partitions of length one or two so we may assume $\len(\lambda)\geq 3$. Suppose $\len(\lambda)=3$, $\lambda=abc$.  If $\lambda$ is  a $p$-core we are done by Corollary 1.6.  If $\lambda$ is not a $p$-core then the edge length $a+2$ is at least $p$ so that $a=p-2$ or $p-1$. If $a=p-2$ then $\lambda$ has the required form. If $a=p-1$  again $\lambda=(p-2,b-1,c-1)+\omega_3$ has the required form. 

\q We suppose that the result is false and that $\lambda$ is a counterexample of minimal degree. Then the length $n$, say, of $\lambda$ is at least $4$.  We set $\hat\lambda=(\lambda_1,\ldots,\lambda_{n-1})$, $\tilde\lambda=(\lambda_2,\ldots,\lambda_n)$.  Thus we have $\hat\lambda=(p-2)^kab+\omega_r$, for some $k\geq 0$ and $p-2>a\geq b\geq 0$, $r\geq 0$, by Proposition  1.11  and the minimality assumption.  
If $k=0$ then $r=\len(\hat\lambda)=n-1$ and $\lambda_{n-1}=1$ so that $\lambda_n=1$ and $\lambda=ab+\omega_n$.  Thus we have $k>0$ and so $\lambda_1=p-1$ or $p-2$. Now $\tilde\lambda=(p-2)^mcd+\omega_s$, for some $p-2>c\geq d\geq 0$ and $m\geq0$ and $s\geq 0$. If $\lambda_1=p-1$ then we have $\lambda=(p-2)^{m+1}cd+\omega_{s+1}$, which has the required form.

\q Thus we may assume $\lambda_1=p-2$ and $\hat\lambda=(p-2)^kab$. Hence we have $\tilde\lambda=(p-2)^mcd$, with $m>0$, $p-2>c\geq d\geq 0$ or $\tilde\lambda=ef+\omega_{n-1}$, $p-2>e\geq f\geq 0$. In the first case we have $\lambda=(p-2)^{m+1}cd$, which has the required form. Thus we may assume that $\tilde\lambda=ef+\omega_{n-1}$.  This gives $\lambda_n=1$. Now if $b=0$ or $a=b=0$ then $\lambda=(p-2)^{n-2}a1$ or $\lambda=(p-2)^{n-1}1$ respectively and both have the required form. Thus we may assume $\lambda=(p-2)^kab1$, with $b>0$. If $p=3$ we must have $\lambda=\omega_n$,  which has the required form. Thus we may assume $p>3$. 

\q We now consider the case $k=1$, so $\lambda=(p-2)ab1$, $p-2> a\geq b\geq 1$. If $b=1$ then $\lambda=(p-3)(a-1)+\omega_4$ which has the right form. Hence, we may assume that $b>1$. Then, $\lambda$ has the removable node $R=(1,p-2)$ with residue $-3$  and lower addable  nodes to be found among  $(2,a+1)$, $(3,b+1)$, $(4,2)$, $(5,1)$,  with residues $a-1$, $b-2$, $-2$, $-4$.  Since $a\neq p-2$ and $b\neq p-1$ we have that $\lambda_R=(p-3)ab1$ must be $(2,1)$-good, by Lemma 1.18    and this must be  of the required form by  minimality  assumption. However, this is not the case since $b>1$ and so we have a contradiction.

\q We now consider the case $k=2$, so $\lambda=(p-2)^2ab1$, $p-2> a\geq b\geq 1$. If  $a=b=1$ then $\lambda=(p-3)^2+\omega_5$ and so it is of the right form. Thus we may assume $a>1$.  Then, $\lambda$ has the removable node $S=(2,p-2)$ with residue $-4$  and lower addable  nodes among  $(3,a+1)$, $(4,b+1)$, $(5,2)$, $(6,1)$,  with residues $a-2$, $b-3$, $-3$, $-5$.  Since $a\neq p-2$ and $b\neq p-1$ we have that $\lambda_S=(p-2)(p-3)ab1$ must be $(2,1)$-good, by Lemma 1.18  and this must be of the required t form by   minimality.  However, this is not the case, since $a>1$,  and so we have a contradiction.

\q Thus we have $k\geq3$ and $\lambda=(p-2)^kab1$. Then, $\lambda$ has the removable node $T=(k,p-2)$ with residue $-k-2$  and lower addable  nodes among $(k+1,a+1)$, $(k+2,b+1)$, $(k+3,2)$, $(k+4,1)$,  with residues $a-k$, $b-k-1$, $-k-1$, $-k-3$.  Since $a\neq p-2$ and $b\neq p-1$ we have that $\lambda_T=(p-2)^{k-1}(p-3)ab1$ must be $(2,1)$-good, by Lemma 1.18,  and this must be of the required  form by minimality.  Once again this is not the case and we have a contradiction, so we are done. 

\end{proof}

\begin{lemma} Let $\lambda$ be a restricted partition with $\lambda_1=p$. Then $\lambda$ is $(2,1)$-good if and only if it has the form $(p-1)^ka+\omega_r$ for $k\geq 1, r\geq 1$, $p-1 >  a\geq 0$.

\end{lemma}

\begin{proof}  Certainly all partitions of this form are $(2,1)$-good, by Lemma 1.18  and Lemma 2.1.

\q We now show that any $(2,1)$-good partition $\lambda$ with $\lambda_1=p$ has the required form.  We assume for a contradiction  that this is not so and that $\lambda$ is a counterexample of minimal degree.  If $\lambda$ has length one then $\lambda$ cannot be restricted. If $\lambda$ has length $2$ then $\lambda=(p,a)$, with $0<a\leq p-1$, so that $\lambda=(p-1,a-1)+(1,1)$.  

\q Let $n=\len(\lambda)$ and define 
$\hat\lambda=(\lambda_1,\ldots,\lambda_{n-1})$ and $\tilde\lambda=(\lambda_2,\ldots,\lambda_n)$.  Note that $\lambda_2<p$ for otherwise, by Lemma 1.11,  we can write $\tilde\lambda=(p-1)^ka+\omega_r$ and  therefore $\lambda=(p-1)^{k+1}+\omega_{r+1}$. 

\q Now assume that $\len(\lambda)=3$ so that $\lambda=pab$, for $p-1\geq  a\geq b\geq 1$.  If $b=1$ then $\lambda=(p-1,a-1)+(1,1,1)$. So we have  $p-1\geq a\geq b\geq 2$. If $p=3$ then $\lambda=322$ which obviously has the required form. So we may assume that $p>3$.  Suppose $b>2$. The  node $R=(3,b)$ is  suitable. However, $\lambda_R=pa(b-1)$ does not have the required form, contrary to  Lemma 1.18 and the minimality assumption.  Hence we have $\lambda=pa2$ and it is easy to see that this is a core so that $\lambda$ is not $(2,1)$-good by Corollary 1.6. 

\q Thus we have $n\geq 4$.  The node $R=(n,\lambda_n)$ is suitable, so that we may write $\lambda_R=(p-1)^ka+\omega_r$, for some $k\geq 1$, $p-1>a\geq 0$, $1\leq r\leq n$. If $r=1$ then $\lambda_n=a+1$ and $\lambda=(p-1)^k(a+1)+\omega_1$, which has the required form. Hence $r>1$ and so $k=1$, since $\lambda_2<p$.  Thus we have $\lambda_R=(p-1)a+\omega_r$, with $r=n-1$ or $r=n$. But $r=n$ is impossible since $\lambda_{n-1}\geq \lambda_n$ so that $\lambda_R=(p-1)a+\omega_{n-1}$ and $\lambda=(p-1)a+\omega_n$, and we are done.

\end{proof}

\q We next treat the non-restricted partitions. 

\begin{lemma} Let $\lambda$ be a non-restricted partition with $\lambda_1\leq 2p-1$. Then $\lambda$ is $(2,1)$-special if and only if it has the form $(p-2)^kab+\omega_r+p\omega_s$, for some $k\geq 0$, $p-2>a\geq b\geq 0$, $r\geq0$ and $s\geq 1$.

\end{lemma}

\begin{proof}  
Certainly all partitions of this form are $(2,1)$-special by Lemma 1.9    since the partitions of the form $(p-2)^kab+p\omega_s$ are $2$-special, by Proposition 2.2   and $\omega_r$ is $(0,1)$-special. 

\q It is easy to see that if $\lambda$ is a  nonrestricted partition of length  at most two and $\lambda_1\leq 2p-1$ then $\lambda$ has the required form. 
Suppose for a contradiction that the lemma is false and that $\lambda=\mu+p\omega_s$ is a partition of smallest length $n$, say, for which it fails and suppose among all partitions of length $n$,  for which if fails, $s$ is as small as possible.  Thus we have  $n\geq 3$.  We now proceed in several steps.

\medskip

{\it  Step 1}  \rm\q  We have $\lambda_1<2p-1$.

\q  Suppose $\lambda_1=2p-1$.  Now $\tilde\lambda=(\lambda_2,\ldots,\lambda_n)$ is $(2,1)$-special, by  Proposition 1.11.   If $\tilde\lambda$ is restricted then $\lambda_1-\lambda_2\geq p$ so that $\lambda_2\leq p-1$ and by Lemma 2.3  we have $\tilde\lambda=(p-2)^kab+\omega_r$, for some $k\geq 0$, $p-2>a\geq b\geq 0, r\geq0$.  But then, $\lambda_1=2p-1$ gives  $\lambda=(p-2)^{k+1}ab+\omega_{r+1}+p\omega_{1}$. If $\tilde\lambda$ is non-restricted then by the minimality of the length we have that $\tilde\lambda=(p-2)^kab+\omega_r+p\omega_s$, for some $k\geq 0$, $p-2>a\geq b\geq 0, r\geq0,s\geq1$. But then, $\lambda_1=2p-1$ gives  $\lambda=(p-2)^{k+1}ab+\omega_{r+1}+p\omega_{s+1}$.

\medskip

{\it  Step 2.}   We have  $s\leq n/2$.

\q  Suppose $s>n/2$.  We apply the reciprocity  principle with respect to $n$, as in Proposition 1.10.

\q Now $\lambda^\dagger$ is $(2,1)$-special. We have 
\begin{align*}\lambda^\dagger&=(2p-1)\omega_n-(\mu_n,\ldots,\mu_1)-p(\omega_n-\omega_{n-s})\cr
&=(p-1)\omega_n-(\mu_n,\ldots,\mu_1)+p\omega_{n-s}\cr
&=(p-1-\mu_n,\ldots,p-1-\mu_1)+p\omega_{n-s}.
\end{align*}
Now $n-s<s$ so by minimality so we have 
$$\lambda^\dagger=(p-2)^kab+\omega_r+p\omega_s$$
for some $r,s\geq 0$.  So  we have 
\begin{align*}\lambda&=(\lambda^\dagger)^\dagger=(2p-1)\omega_n-(0,\ldots,0,b,a,p-2,\ldots,p-2)\cr
&\phantom{something} -(\omega_n-\omega_{n-r})-p(\omega_n-\omega_{n-s})\cr
&=(p-2)\omega_n-(0,\ldots,0,b,a,p-2,\ldots,p-2)+\omega_{n-r}+p\omega_{n-s}\cr
&=(p-2,\ldots,p-2,p-2-b,p-2-a,0,\ldots,0)+\omega_{n-r}+p\omega_{n-s}
\end{align*}
which has the required form.

\medskip

{\it   Step 3.}    We have $n\geq 4$.

\q If not, we have $n=3$ and $s=1$, by Step 2. Thus we have $\lambda=(a+p,b,c)$ with $a\geq b\geq c\geq 0$ and $a\leq p-2$, by Step 1.  We may remove a $p$-hook from the first
 row of $\lambda$ so the $p$-core of $\lambda$ is the $p$-core of $(a,b,c)$. Now the length of the edge of the diagram of $(a,b,c)$ is $a+2\leq p$. If $a=p-2$ then 
 $(p+a,b,c)=(p-2,b,c)+p\omega_1$ and $\lambda$ is no counterexample. So we may assume that $a+2<p$ so that $(a,b,c)$ is a $p$-core. Thus we have $c\leq 1$, by Corollary 1.6 (i).    If $c=0$ then $(p+a,b,c)=(a,b)+p\omega_1$ and $\lambda$ is no counterexample. If $c=1$ then $(p+a,b,c)=(a-1,b-1)+\omega_3+p\omega_1$ and $\lambda$ is no counterexample.

 \medskip

{\it Step 4.} \q\rm Conclusion.
 
 \q Consider $\hat\lambda=(\lambda_1,\ldots,\lambda_{n-1})$. Then $\hat\lambda$ is $(2,1)$-special, by Proposition  1.11,  so by minimality of length we have $\hat\lambda=(p-2)^kab+\omega_m+p\omega_s$, for some $k,m\geq 0$, $s\geq 1$ and $p-2>a\geq b\geq 0$.  By Steps 2 and 3 we have $s<n-1$ so that the length $n-1$ of $\hat\lambda$ is either $m$ or the length of $(p-2)^kab$. Moveover, $k=0$ or $m=0$ since $\lambda_1<2p-1$ (by Step 1).  If $k=0$ then $m=n-1$ and $\lambda_{n-1}=1$ which implies that $\lambda_n=1$ and hence $\lambda=ab+\omega_n+p\omega_s$. So we can assume $m=0$. Thus we have $\hat\lambda=(p-2)^kab+p\omega_s$, with $k>0$ (since $\lambda$ has length at least $4$ by Step 3).  If $b=0$ then $\lambda$ is either $(p-2)^{n-1}\lambda_n+p\omega_s$ or $(p-2)^{n-2}a\lambda_n+p\omega_s$ and $\lambda$ is no counterexample.  Thus we can assume $b>0$ and $\lambda=(p-2)^{n-3}ab\lambda_n+p\omega_s$ and 
 $$\tilde\lambda=(\lambda_2,\ldots,\lambda_n)=(p-2)^{n-4}ab\lambda_n+p\omega_{s-1}.$$
 
 But $\tilde\lambda$ is $(2,1)$-special . But this contradicts Lemma 2.3  if $s=1$ and the minimality of $n$ if $s>1$.

\end{proof}

\begin{proposition} Let $\lambda$ be a non-restricted partition with $\lambda_1\leq 2p-1$.  The following are equivalent:

(i) $\lambda$ is $(2,1)$-special;

(ii)  $\lambda$ is $(2,1)$-good;

(iii)    $\lambda$ has the form $(p-2)^kab+\omega_r+p\omega_s$, for some $k\geq 0$, $p-2>a\geq b\geq 0$ and $s\geq 1$.

\end{proposition}

\begin{proof}  Certainly (i) implies (ii).  Moreover,  we have already seen that a partition of the form $(p-2)^kab+p\omega_s$, with $k\geq 0$, $p-2>a\geq b\geq 0$, $s\geq 1$ is $(2,0)$-special. But $\omega_r$ is $(0,1)$-special, for $r\geq 0$,  and hence $(p-2)^kab+p\omega_s+\omega_r$ is $(2,1)$-special.   Thus (iii) implies (i). We also have, by Lemma 2.5, that (i) implies (iii). 

\q So it remains to prove that (ii) implies (i). We assume that $\lambda$ is $(2,1)$-good. Then $L(\lambda)$ is a composition factor of $L(\theta)\otimes L(\tau)^F$, for some $(2,1)$-special partition $\theta$ and some $2$-good partition $\tau$, by  Proposition 1.4.  We write $\theta=\alpha+p\beta$, for partitions $\alpha$ and $\beta$ with $\alpha$ restricted. 

\q  If $\beta=0$ then $\lambda=\theta+p\tau$. Also, we have $\tau_1\leq 1$ since $\lambda_1\leq 2p-1$.  If $\tau_1=0$ then $\lambda=\theta$, which is $(2,1)$-special. If $\tau_1=1$ then $\theta_1\leq p-1$ and so $\lambda$ is $(2,1)$-special by Lemmas 2.3 and 2.5.

\q   If $\beta\neq 0$ then $\theta$ has the form $(p-2)^kab+\omega_r+p\omega_s$, for some $k\geq 0$, $p-2>a\geq b\geq 0, r\geq0,s\geq1$ by  Lemma 2.5.   Moreover, $L(\lambda)$ is isomorphic to $L(\alpha)\otimes L(\gamma)^F$ for some partition $\gamma$ such that  $L(\gamma)$ is a composition factor of $L(\beta)\otimes L(\tau)$.  Since $\lambda_1\leq 2p-1$ we must have $\gamma_1\leq 1$ so that $\gamma=\omega_t$ for some $t\geq 1$ and we are done.

\end{proof}

 \q We continue with the analysis of restricted $(2,1)$-good partitions.

\begin{lemma}   Let $\lambda$ be a restricted partition with $p<\lambda_1< 2p-2$. Then $\lambda$ is $(2,1)$-good if and only if it has the form $(p-1)^ka+(b)+\omega_r$, for some $k\geq 1$, $p-1>a\geq 0$, $p-1>b>  0$, $r\geq 0$, with $r>0$ if $b=1$.

\end{lemma}

\begin{proof} Certainly all partitions of the  given form are $(2,1)$-good, by Lemma 1.9  and Lemma 2.1.

\q   There are no partitions of length one satisfying the hypotheses. If  $\lambda$ has length $2$  then $\lambda=(p+b,c)$, with $0<b<c\leq p-1$ and we may write $\lambda=(p-1,c)+(b)+\omega_1$. 
 
\q  We now consider the case in which $\lambda$ has length $3$.   Then $\lambda=(p+b)cd$ with $1\leq b<c\leq p-1$ and $1\leq d\leq c\leq p-1$ or $\lambda=(p+b)(p+c)d$ with $0\leq c\leq b<p-2$, $b\neq 0$  and $0\leq c<d\leq p-1$. 

\q 
Assume first that $\lambda=(p+b)cd$. If $d=1$ then $\lambda=(p-1)(c-1)+(b)+\omega_3$ and so  has the right form. If $c=p-1$ then $\lambda=(p-1)^2d+(b)+\omega_1$ and  again has the  right form.    Hence we may assume that  $2\leq d\leq c\leq p-2$ and $1\leq b\leq c\leq p-2$. If $\lambda$ is a $p$-core then since $\lambda_3=d\geq2$ we get that $\lambda$ is not a $(2,1)$-special partition, by Corollary 1.6. Therefore, we may assume that $\lambda$ is not a $p$-core. Hence we can remove at least one $p$-hook from the diagram of $\lambda$. Since $c\leq p-2$ and $b<c$ this $p$-hook must involve boxes from the first two rows or boxes from all the three rows of $\lambda$. If we can remove a $p$-hook involving boxes only from the first two rows of $\lambda$ then the resulting  partition is  $(c-1,b+1,d)$ and since $c\leq p-2$ this is a core, so that $\core(\lambda)=(c-1,b+1,d)$.  Again, since $d\geq2$ we get that $\lambda$ is not a $(2,1)$-special partition, by Corollary 1.6.  If the $p$-hook that we remove involves boxes from all the rows of $\lambda$ then the resulting partition has the form $(c-1,d-1,b+2)$ and since $c\leq p-2$ this is a core, so that $\core(\lambda)=(c-1,d-1,b+2)$. But this has third part at least $2$ and so is not $(2,1)$-good, by Corollary 1.6. 

\q  We now consider  the situation in which  $\lambda=(p+b)(p+c)d$ with $0\leq c<b<p-2$ and $0\leq c<d\leq p-1$. If $c=0$ then $\lambda=(p-1)^2d+(b)+\omega_2$ and so has the right form. Hence, we may assume that $c>0$ (so $d>1$) and we prove that $\lambda$ is not $(2,1)$-special. Assume for a contradiction that it is. Then by the reciprocity principle, Proposition 1.10,   we have that $\lambda^\dagger=(2p-1)\omega_3-w_0\lambda=(2p-1-d)(p-1-c)(p-1-b)$ is $(2,1)$-special. If $d=p-1$ then $\lambda^\dagger=(p,p-1-c,p-1-b)$ with $1<p-1-b\leq p-1-c<p-1$ and so we get a contradiction by Lemma 2.4.   Thus, we may assume that $1<d<p-1$ and so $\lambda^\dagger$ is a restricted partition of the form $(p+e)fg$ with $1\leq e<p-2$. Hence, by the previous paragraph,  it can be written in the form $(p-1)^ka+(h)+\omega_r$ for $k\geq1, h\geq1,r\geq1, 0\leq a<p-1.$ Since $p-1-c<p-1$ we have that the length of $\lambda$ is achieved by $\omega_r$ and so $r=3$. In particular we have  $p-1-b=1$ and so $b=p-2$ which contradicts the assumption on $b$. Therefore, $\lambda$ is not $(2,1)$-special and we are done.

\q  Now assume, for a contradiction, that the result  is false, and let $\lambda$ be a counterexample of minimal degree. Note that $\lambda$ has length at least $4$ by the above.  We set $\hat\lambda=(\lambda_1,\ldots,\lambda_{n-1})$ and $\tilde\lambda=(\lambda_2,\ldots,\lambda_n)$.  We divide the rest of the proof up into four cases.

\medskip

\newpage

{\it  Case 1.}  $\lambda_n=1$ and $\lambda_{n-1}\neq p$.

 \q We have $\hat\lambda=(p-1)^ka+(b)+\omega_r$ for some $k\geq 1$, $0\leq a\leq p-2$, $0<b\leq  p-2$ and $r\geq 0$.  If $r=n-1$ then $\lambda=(p-1)^ka+(b)+\omega_{n}$ has the required form, so we assume $r<n-1$. Hence the length $n-1$ of $\hat\lambda$ is achieved from the partition $(p-1)^ka$ so that $k=n-2$, $0<a<p-1$ or $k=n-1$.

  \q Suppose that $k=n-2$ so that $\lambda=(p-1)^{n-2}a1+(b)+\omega_r$. If $r=n-2$ then $\lambda=(p-1)^{n-2}(a-1)+(b)+\omega_n$, which has the required form. If $r=n-3$ and $a=1$ then $\lambda=(p-1)^{n-3}(p-2)+(b)+\omega_n$, which again has the required form. So we may assume $r<n-2$ and it is not the case that $r=n-3$ and $a=1$. 
Now $\lambda=(p-1)^{n-2}a1+(b)+\omega_r$ has the removable node $R=(n-2,p-1)$ with  residue $-n+1$ and lower addable nodes $(n-1,a+1)$, $(n,2)$, $(n+1,1)$ with residues $-n+a+2$, $-n+2$, $-n$. Since $a\neq p-1$ we get, from Lemma 1.18, that $\lambda_R=(p-1)^{n-3}(p-2)a1+(b)+\omega_r$ has the required form. But this happens only in the case $a=1$, $r=n-3$, which we have already dealt with. 

\q Suppose $k=n-1$. Then we have $\hat\lambda=(p-1)^{n-1}+(b)+\omega_r$ and $r<n-1$. Then $\lambda=(p-1)^{n-1}1+(b)+\omega_r$ has the required form.

\medskip

{\it  Case 2.}    $\lambda_n=1$ and $\lambda_{n-1}=p$.

  \q   Then $\lambda$ has the removable node $S=(n-1,p)$, which has residue $-n+1$ and lower addable nodes $(n,2)$ and $(n+1,1)$, with residues $-n+2$ and $-n$. From Lemma 1.18  we get that $\lambda_S$ is $(2,1)$-good and so we may write
   $\lambda_S=(p-1)^ka+(b)+\omega_r$, with $k\geq 1$ 
  with  $p-1>a\geq 0$, $p-1>b>0$, $r\geq 0$. Since $\lambda_{n-1}=p$ we must have $k\geq n-2$. 

\q Suppose $k=n-2$. Then $a=p-2$, $r=n$ so that 
$$\lambda_S=(p-1)^{n-2}(p-2)+(b)+\omega_n$$
 and $\lambda=(p-1)^{n-1}+(b)+\omega_n$ has the required form.

\q Suppose $k=n-1$. Then $\lambda_S=(p-1)^{n-1}a+(b)+\omega_r$.  We note that $r<n-1$ since $\lambda_{n-1}=p$.  Hence $a=\lambda_n=1$, $\lambda_S=(p-1)^{n-1}1+(b)+\omega_r$. If $r=n-2$ we get $\lambda=(p-1)^{n-1}+(b)+\omega_n$, which is of the required form. However, we can not have $r<n-2$ since then $\lambda_S=(p-1)^{n-1}1+(b)+\omega_r$ would give $\lambda_{n-2}=p-1$, $\lambda_{n-1}=p$.

\medskip

{\it  Case 3.}     $\lambda_n>1$ and, for $T=(n,\lambda_n)$,  the partition $\lambda_T$ is restricted. 

\q In this case we have $\lambda_T=(p-1)^ka+(b)+\omega_r$, for some $k\geq 1$, $p-1>a\geq 0$, $0<b<p-1$ and $r\geq 0$. Note that we have $\lambda_{n-1}\geq 2$ and hence $k\geq n-2$. 

\q Suppose $k=n-2$. Then we have $r=n$ and  $\lambda=(p-1)^{n-2}a1+(b)+\omega_n$  If $a=p-2$ we can write $\lambda=(p-1)^{n-1}2+(b)+\omega_{n-2}$, in the required form.  So we may assume $a<p-2$. Now $\lambda$ has the removable node $U=(n-2,p)$ with residue $-n+2$ and lower addable nodes $(n-1,a+2)$, $(n,3)$ and $(n+1,1)$ with residues $-n+a+3$, $-n+3$ and $-n$. If follows from Lemma 1.18  that $\lambda_U=(p-1)^{n-3}(p-2)a1+(b)+\omega_n$ has the required form. But this is so only if $a=p-2$, a case we have already excluded. So we have a contradiction.

\q Suppose now that $k=n-1$ so that $\lambda_T=(p-1)^{n-1}a+(b)+\omega_r$. Then we have $\lambda=(p-1)^{n-1}(a+1)+(b)+\omega_r$, which has the required form.

\q Suppose finally  that $k=n$. Then $\lambda_T$ has the form $(p-1)^n+(b)+\omega_r$. But then $\lambda_n$ is at least $p$ and this is not possible since $\lambda$ is restricted.

\medskip

{\it Case  4.}    $\lambda_n>1$ and for $T=(n,\lambda_n)$ the partition $\lambda_T$ is non-restricted.

\q We must have that $\lambda_{n-1}=\lambda_n+p-1$. Then $\lambda$ has the removable node  $V=(n-1,\lambda_{n-1})$ with residue $-n+\lambda_n$ and lower addable nodes $(n,\lambda_n+1)$ and $(n+1,1)$ with residues $-n+\lambda_n+1$ and $-n$. From Lemma 1.18  we get that $\lambda_V$ is expressible in the form $(p-1)^ka+(b)+\omega_r$.  But then we get $\lambda_{n-2}\leq p$ and $\lambda_{n-1}=\lambda_n+p-1>p$, a contradiction.

\q This completes the proof.

\end{proof}

\begin{lemma} Let $\lambda$ be a restricted $(2,1)$-good partition with $\lambda_1=2p-2$. Then $\lambda$ has the form $(p-1)^ka+(p-1)^mb$, with $k,m\geq 1$, $p-1>a\geq 0$, $p-1>b\geq 0$, or the form $(p-1)^ka+(p-2)+\omega_r$, with $k\geq 1$ and $p-1>a\geq 0$, $r\geq 1$.

\end{lemma}

\begin{proof} As usual we remark that by earlier results, Lemma 1.9  and Lemma 2.1,  all  partitions  of the given form are $(2,1)$-good.

\q We now suppose that $\lambda$ is $(2,1)$-good  restricted partition with $\lambda_1=2p-2$ and we show that $\lambda$ has the form. This is vacuously true for partitions of length one. Moreover, if $\lambda$ has length $2$  then $\lambda=(2p-2,p-1)=(p-1,p-1)+(p-1)$ and so  is of the right form. So, we now  assume $\lambda$ has length at least $3$.  We write $n$ for the length of $\lambda$.  

\q Assume first that $\lambda_n>1$. Then we consider the partition $\lambda^\dagger$. Since $\lambda_1=2p-2$ we have that  $\lambda^\dagger$  is restricted. Moreover $\lambda^\dagger$  has first entry less than $2p-2$ and  $\lambda^\dagger$ is $(2,1)$-special by Proposition 1.10.  The first entry of $\lambda^\dagger$ is $2p-1-\lambda_n$ and this is at least $p$ as $\lambda_n<p$ (as $\lambda$ is restricted).

\q Then, by Lemmas 2.4 and 2.7  we have $\lambda^\dagger=(p-1)^ka+(b)+\omega_r$, with $k\geq 1$, $p-1>a\geq 0$, $b,r\geq 0$ and either $b$ or $r$ nonzero. Since $\lambda=(\lambda^\dagger)^\dagger$ we have 
\begin{align*}\lambda&=(2p-1)\omega_n-(0,0,\ldots,a,p-1,\ldots,p-1)-(0,0,\ldots,b)-(\omega_n-\omega_{n-r})\cr
&=(p-1,\ldots,p-1-a,0,\ldots,0)+(p-1,\ldots,,p-1,p-1-b)+\omega_{n-r}
\end{align*}

and, since $\lambda_1=2p-2$, we easily conclude that $\lambda$   has the required form. 

\q  Now suppose that the result  is false and that $\lambda$ is a counterexample of minimal degree. We know that the length $n$ of $\lambda$ is at least $3$ and $\lambda_n=1$.  

\q Suppose $n=3$. Then we have either
$$\lambda=(2p-2,p,1)=(p-1)^2+(p-2)+\omega_3$$
or
$$\lambda=(2p-2,p-1,1)=(p-1)^21+(p-1)$$
and $\lambda$ has the required form.  

\q Now suppose the length of $n=4$.  Thus we may write $\lambda=(2p-2,p-1+a,b,1)$, with $0\leq a\leq p-1$ and $p\geq b\geq 1$. If $a=p-1$ then $b=p$ or $p-1$ (since $\lambda$ is restricted) and then
$$\lambda=(2p-2,2p-2,p,1)=(p-1)^31+(p-1)^21$$
or
$$\lambda=(2p-2,2p-2,p-1,1)=(p-1)^31+(p-1)^2$$
and so has the required form.  So we have $p-1>a$. So $\lambda$ has the removable node $R=(1,2p-2)$ which has residue $-3$ and lower addable nodes $(2,p+a)$, $(3,b+1)$, $(4,2)$, $(5,1)$ with residues $a-2$, $b-2$, $-2$, $-4$.  Either we may apply Lemma 1.18 or $b=p-1$. In the latter case we have
$$\lambda=(2p-2,p-1+a,p-1,1)=(p-1)^31+(p-1)a$$
which has the required form. So we map apply Lemma 1.18 to get that 
$$\lambda_R=(2p-3,p-1+a,b,1)$$
is $(2,1)$-good. By Lemma 2.4 for $p=3$ and Lemma 2.7 for the general case we get that this is only possible if $a=0,b=1$ or $a=1,1\leq b\leq p$. For the first case we get that $\lambda_R=(2p-3,p-1,1,1)$ and so 
$$\lambda=(2p-2,p-1,1,1)=(p-1,p-2)+(p-2)+\omega_4$$
 and so it has the right form. In the  second case we have that  $\lambda_R=(2p-3,p,b,1)$  and so 
 $$\lambda=(2p-2,p,b,1)=(p-1,p-1,b-1)+(p-2)+\omega_4$$
  and again  has  the required  form.  Thus we have $n\geq 5$.   We divide the remainder of the proof into two  cases. We set $\hat\lambda=(\lambda_1,\ldots,\lambda_{n-1})$. 

\medskip

{\it Case 1.}    $\hat\lambda$ is restricted.

\q There are two possibilities:  (i) $\hat\lambda=(p-1)^ka+(p-1)^lb$, for some $k,l\geq 1$, $p-1>a,b\geq 0$; and (ii) $\hat\lambda=(p-1)^mc+(p-2)+\omega_r$, for some $m\geq 1$, $p-1>c\geq 0$, $r\geq 1$. We consider these separately.

\q Suppose we are in the situation (i). We may assume $k\geq l$. If $k, l\geq 3$ then we must have $\tilde\lambda=(p-1)^ud+(p-1)^ve$ for some $u,v\geq 1$ and then $\lambda=(p-1)^{u+1}d+(p-1)^{v+1}e$ has the required form. Suppose now that $k\geq 3$, $l=1$. Then $k=n-2$ or $n-1$. If $k=n-1$ then $\lambda=(p-1)^{n-1}1+(p-1)b$, which has the required form. So we assume $k=n-2$ so that $\lambda=(p-1)^{n-2}a1+(p-1)b$ with  $a=\lambda_{n-1}$. We first suppose that $n=5$. Then $\lambda=(p-1)^3a1+(p-1)b$. If $a=b=1$ then 
$$\lambda=(2p-2,p,p-1,1,1)=(p-1)^2(p-2)+(p-2)+\omega_5$$  and  has the required form. Hence we may assume that $a,b$ are not simultaneously equal to 1. Then,  $\lambda$ has the removable node $R=(3,p-1)$, with residue $-4$ and lower addable nodes $(4,a+1)$, $(5,2)$, $(6,1)$, with residues $a-3$, $-3$, $-5$. Thus we may apply Lemma 1.18 to deduce that $\lambda_R=(p-1)^2(p-2)a1+(p-1)b$ is $(2,1)$-good and so by minimality of the degree of $\lambda$ it must have the required form. This happens only if $a=b=1$ and  we have already excluded  this case. We now suppose that $n>5$. Thus, $\lambda$ has the removable node $S=(n-2,p-1)$, with residue $-n+1$ and lower addable nodes $(n-1,a+1)$, $(n,2)$, $(n+1,1)$, with residues $-n+a+2$, $-n+2$, $-n$. Thus we may apply Lemma 1.18 to deduce that $\lambda_S=(p-1)^{n-3}(p-2)a1+(p-1)b$ is $(2,1)$-good. But it does not have the required form, contradicting the minimality of the degree of $\lambda$.

\q This leaves the possibility  $l=2$,  which we consider now.  Thus we have $\hat\lambda=(p-1)^{n-2}a+(p-1)^2b$ (again as above we may exclude the case $k=n-1$ for  then $\lambda$ has the required form). We first suppose $n=5$. So, $\lambda=(p-1)^3a1+(p-1)^2b$. If $p=3$ then $\lambda=(4,4,3,1,1)$ or $\lambda=(4,4,2,1,1)$. For the first case we have that $\core(\lambda)=(4,2,2,1,1)$ and for the second that $\core(\lambda)=(3,2,2,1,1)$ and so in either case $\lambda$ is not $(2,1)$-special. Thus, we may assume that $p>3$. We have the removable node $T=(2,2p-2)$ with residue $-4$ and lower addable nodes $(3,p+b)$, $(4,a+1)$, $(5,2)$, $(6,1)$ with residues  $-3+b$, $a-3$, $-3$, $-5$. Applying Lemma 1.18 we get that $\lambda_T=(2p-2,2p-3,p-1+b,a,1)$ is $(2,1)$-special. This does not have the required form  and we have a contradiction to the minimality of $\lambda$. We now suppose $n>5$.  We write $\lambda=(p-1)^{n-2}a1+(p-1)^2b$. Then $\lambda$ has the removable node $U=(n-2,p-1)$ with residue $-n+1$ and lower addable nodes $(n-1,a+1)$, $(n,2)$, $(n+1,1)$ with residues $-n+a+2$, $-n+2$, $-n$.  By Lemma 1.18 , $\lambda_U=(p-1)^{n-3}(p-2)a1+(p-1)^2b$ is $(2,1)$-special and so must have the required form. This happens only if $r=n-3$ and $c=1$ but then $\lambda=(p-1)^n-3(p-2)+(p-2)+\omega_n$ and  has also the required form.  So we are done.

\q We now move on to possibility  (ii).  So $\hat\lambda=(p-1)^mc+(p-2)+\omega_r$.  If $r=n-1$ then $\lambda=(p-1)^mc+(p-2)+\omega_n$ has the required form.  Suppose now that $r=n-2$ then  $\hat\lambda=(p-1)^{n-2}c+(p-2)+\omega_{n-2}$ and we must have $c=\lambda_{n-1}>0$. Hence we have $\lambda=(p-1)^{n-2}(c-1)+(p-2)+\omega_n$ and $\lambda$ has the required form. So we have  $\lambda=(p-1)^{n-2}c1+(p-2)+\omega_r$, with $r<n-2$. Thus $\lambda$ has the removable node $V=(n-2,p-1)$ with residue $-n+1$ and addable nodes $(n-1,c+1)$, $(n,2)$, $(n+1,1)$ with residues $-n+c+2$, $-n+2$, $-n$. Applying Lemma 1.18 we get that $\lambda_V=(p-1)^{n-3}(p-2)c1+(p-2)+\omega_r$ is $(2,1)$-special. But this is not of the required form,  contradicting  the minimality of $\lambda$.

\medskip

{\it  Case 2.}  $\hat\lambda$ is not restricted.

\q We must have $\lambda_{n-1}=p$. Thus $\lambda$ has the removable node $W=(n-1,p)$ with residue $-n+1$ and lower addable nodes $(n,2)$, $(n+1,1)$ with residues $-n+2$, $-n$. Applying Lemma 1.18 we get that $\lambda_W$ is $(2,1)$-good. There are two possibilities:  $\lambda_W=(p-1)^ka+(p-1)^lb$, with $k\geq l \geq 1$, $p-1>a,b\geq 0$;  $\lambda_W=(p-1)^mc+(p-2)+\omega_r$, with $m\geq 1$, $r\geq 1$. We consider these cases separately. 

\q In the first case,  from $\lambda_{n-1}=p$,  it follows that $\lambda_W=(p-1)^{n-1}1+(p-1)^{n-2}$ and that $\lambda=(p-1)^{n-1}1+(p-1)^{n-2}1$, which has the required form.

\q In the second case  we have $\lambda_W=(p-1)^mc+(p-2)+\omega_r$.  Since $\lambda_{n-2}\geq\lambda_{n-1}=p$ we must have $m,r\geq n-2$. Also, $m\neq n$, since $\lambda_n=1$.   If $m=n-2$ then $\lambda=(p-1)^{n-2}(p-2)+(p-2)+\omega_n$, which has the required form.   If $m=n-1$ then $r=n-2$ (since $\lambda_{n-1}=p$) so that $\lambda_W=(p-1)^{n-1}c+\omega_{n-2}$ and $\lambda=(p-1)^{n-1}c+\omega_{n-1}$, which has the required form.

\end{proof}

\begin{lemma} Let $\lambda$ be a restricted $(2,1)$-good partition with $\lambda_1=2p-1$. Then $\lambda$ has the form $(p-1)^ka+(p-1)^mb+\omega_r$, with $k,m\geq 1$, $p-1>a\geq 0$, $p-1>b\geq 0$.

\end{lemma}

\begin{proof}  Since $\lambda$ is restricted its length $n$ is at least $2$. Suppose the result is false and that $\lambda$ is a counterexample of minimal length. Consider $\tilde\lambda=(\lambda_2,\ldots,\lambda_n)$.  We have $\lambda_2\geq p$ (since $\lambda_1=2p-1$ and $\lambda$ is restricted). Hence, by Lemmas 2.4, 2.7, 2.8 and the minimality of length,  we have $\hat\lambda=(p-1)^ka+(p-1)^mb+\omega_r$ for some $k,m,r\geq 0$ and $p-1>a\geq 0$, $p-1>b\geq 0$ and hence $\lambda=(p-1)^{k+1}a+(p-1)^{m+1}b+\omega_{r+1}$.

\end{proof}

\q Most of the preceding results can be summarised in  the following simple statement. We assume now that $p$ is any prime.

\begin{theorem} A partition $\lambda$ is $(2,1)$-special if and only if $\lambda_1\leq 2p-1$ and $\lambda$ has the form 
$\mu+\omega_s$ for some $2$-special partition $\mu$ and some $s\geq0$.

\end{theorem}

\begin{proof} For $p>2$ this follows from Lemmas 2.3 to 2.9 and Remark 1.3. 

\q Suppose now $p=2$.  Suppose that $\lambda$ is a $(2,1)$-special partition. Then   $L(\lambda)$ is a composition factor of $L(\mu)\otimes \tbw^s E$ for some  $2$-special partition $\mu$ and $s\geq 1$. But then $\mu_1\leq 2p-2=2$, by Remark 1.3.  Now $L(\mu)\otimes \tbw^s E$  has unique highest weight $\mu+\omega_s$ and so $\lambda\leq \mu+\omega_s$, which gives $\lambda_1\leq 2p-2+1=3$.  But now the fact that $\lambda_1\leq 3$ gives that $\lambda$ may be written $\omega_a+\omega_b+\omega_c$ for  some $a,b,c\geq 0$.  Moreover, $\omega_a+\omega_b$ is $2$-special, e.g., by \cite{DG4}, Theorem 6.5 so that $\lambda$ has the required form.

\q Conversely, suppose that $\lambda$ is a partition of the form $\mu+\omega_s$, for some $2$-special partition $\mu$ and $s\geq 0$.   Then $\lambda$ is $(2,1)$-special by Lemma 1.9.

\end{proof}

\subsection*{Symmetric Groups}

\q We can use the analysis above to give a precise description of the simple modules for symmetric groups that appear as composition factors of  Specht modules $\Sp(\mu)$, with $\mu$ a $(2,1)$-bounded partition. For the relation between the polynomial representations of $GL_n(K)$ and the representations of the symmetric groups via the Schur functor  we refer the reader to chapter 6 of \cite{EGS}. Here we consider partitions of degree $r$ with $r\leq n$.

\q We set $D_\lambda=f L(\lambda)$, where $f$ is the Schur functor, for $\lambda$ restricted.  For $\lambda$ a regular partition we write $D^\lambda$ for the head of $\Sp(\lambda)$.  The relationship between the two labelings of the irreducible modules for symmetric groups is given by $K_{\rm sgn} \otimes D^\lambda=D_{\lambda'}$  for $\lambda$ regular and $K_{\rm sgn}$ the sign module, by \cite{EGS},  (6.4l). Moreover for a partition $\mu$ we have that $f \nabla(\mu)=\Sp(\mu)$, see \cite{EGS} section 6.3, and by applying the Schur functor to  a composition series of  $\nabla(\mu)$ one obtains  $[\nabla(\mu):L(\lambda)]=[f\nabla(\mu):fL(\lambda)]=[\Sp(\mu):D_\lambda]$, for a restricted partition $\lambda$.

\begin{corollary} Let $\lambda$ be a restricted partition. Then $D_\lambda$ occurs as composition factor of a Specht module  $\Sp(\mu)$, for some $(2,1)$-bounded partition $\mu$, if and only if  $\lambda$ can be written in the form $(p-2)^kab+\omega_s$ for some $k\geq 0$, $p-2>a\geq b\geq 0$, $s\geq 0$  or $(p-1)^ka+(p-1)^mb+\omega_s$, with $p-1>a\geq 0$, $p-1>b\geq 0$, $s\geq 0$. 

\end{corollary}

\begin{corollary} Let $\lambda$ be a regular  partition. Then $D^\lambda$ occurs as composition factor of a Specht module  $\Sp(\mu)$, for some $(1,2)$-bounded partition $\mu$, if and only if $\lambda'$ can be written in the form $(p-2)^kab+\omega_s$ for some $k\geq 0$, $p-2>a\geq b\geq 0$, $s\geq 0$  or $(p-1)^ka+(p-1)^mb+\omega_s$, with $p-1>a\geq 0$, $p-1>b\geq 0$, $s\geq 0$.

\end{corollary}

\bs\bs\bs



\section{The $2$-good partitions}

\q In this section we describe the $2$-good partitions. We work over an algebraically closed field $K$ of arbitrary positive characteristic $p$.

\begin{definition}

(i) A partition will be called a {\it beginning term} if it has the form $(p-2)^kab$, for some $k\geq 0$, $p-2 \geq  a\geq b\geq 0$.

(ii) A partition will be called a {\it middle  term} if it is not a beginning term but has the form $\lambda+ \omega_r$ for some beginning term $\lambda$ and $r\geq 1$.

(iii) A partition will be called an  {\it end term} if  it is restricted, not $2$-special and can be written in the form $\lambda+\omega_r$ for some  $2$-special partition $\lambda$  and $r\geq 1$. 
\end{definition}

\begin{definition}
We call a partition $\lambda$   {\it primitive} if either: 

(i) $\lambda$ is a  restricted $2$-special partition; or

(ii) $\lambda$  can be written of the form,
 $\lambda=\lambda^0+p\lambda^1+\dots+p^m\lambda^m$ for some $m>0$ and partitions $\lambda^0,\ldots,\lambda^m$ such that $\lambda^0$ is a beginning term, $\lambda^1,\ldots,\lambda^{m-1}$ are middle terms and $\lambda^m$ is an end term.

\q We will say that the primitive partition is  of index $0$ in   the first case  and of index $m$ in  the second.

\end{definition}

\begin{definition}

 We define the set of standard partitions to be the smallest set of partitions such that:

(i) all primitive partitions are standard; and 

(ii) a partition  $\lambda$ is standard if it has the form $\mu+p^{m+1}\tau$ for some primitive partition $\mu$ of index $m$ and a standard partition $\tau$.

\end{definition}

Thus a standard partition $\lambda $  has the form 
$$\lambda=\lambda(0)+p^{m_0+1}\lambda(1)+\cdots+p^{m_0+\cdots+m_{k-1}+k}\lambda(k)\eqno(*)$$
for some $k\geq 0$, where each $\lambda(i)$ is a primitive partition and where $m_i$ is the index of $\lambda(i)$ for $0\leq i<k$.

\begin{remark} Note that,  by  \cite{DG4},  Theorem 6.5,  any  $2$-special partition is standard.
\end{remark}

 We check that standard partitions have the following unique readability property.

\begin{lemma}  Suppose for a partition $\lambda\neq 0$ we have 
\begin{align*}\lambda&=\mu(0)+p^{m_0+1}\mu(1)+\cdots+p^{m_0+\cdots+m_{k-1}+k}\mu(k)\cr
&=\tau(0)+p^{n_0+1}\tau(1)+\cdots+p^{n_0+\cdots+n_{l-1}+l}\tau(l)\
\end{align*}

where all $\mu(i),\tau(j)$ are primitive, $\mu(k)\neq 0\neq \tau(l)$ and  $\mu(i)$ has index $m_i$ and $\tau(j)$ has index $n_j$, for $0\leq i<k$, $0\leq j<l$.

\q  Then we have   $k=l$ and $\mu(i)=\tau(i)$ for $0\leq i\leq k$.
\end{lemma}

\begin{proof}  We consider the base $p$ expansion 
$$\lambda=\sum_{h=0}^N p^h\lambda^h$$
 (where each  $\lambda^h$ is  restricted  and $\lambda^N\neq 0$). 
 
 \q If $\lambda^0$ is not a beginning term then we must have $\mu(0)=\lambda^0=\tau(0)$ and 
 
 \begin{align*}
 \mu(1&)+p^{m_1+1}\mu(2)+\cdots+p^{m_1+\cdots+m_{k-1}+(k-1)}\mu(k)\cr
 &=\tau(1)+p^{n_1+1}\tau(2)+\cdots+p^{n_1+\cdots+n_{l-1}+(l-1)}\tau(l)
 \end{align*}
 and we obtain inductively that $k-1=l-1$, so that $k=l$, and $\mu(i)=\tau(i)$, for all $i$.
 
 \q Thus we may assume that $\lambda^0$ is a beginning term. Suppose that $\lambda^h$ is also a beginning term for some  $0<h\leq N$.  Then for some $0\leq s\leq k$, $0\leq t\leq l$, we have

 \begin{align*}\lambda^0+p\lambda^1+\cdots+p^{h-1}\lambda^{h-1}&=\mu(0)+p^{m_0+1}\mu(1)+\cdots+p^{m_0+\cdots+m_{s-1}+s}\mu(s)\cr
&=\tau(0)+p^{n_0+1}\tau(1)+\cdots+p^{n_0+\cdots+n_{t-1}+t}\tau(t).
\end{align*}

Thus we have inductively $s=t$, $\mu(i)=\tau(i)$ for $0\leq i \leq s$.  ŒIn particular we have $\mu(0)=\tau(0)$ and $m_0=n_0$. Subtracting $\mu(0)=\tau(0)$ from $\lambda$ and dividing by $p^{m_0+1}$ we thus obtain
\begin{align*}\mu(1)&+p^{m_1+1}\mu(2)+\cdots+p^{m_1+\cdots+m_{k-1}+k-1}\mu(k)\cr
&=\tau(1)+p^{n_1+1}\tau(2)+\cdots+p^{n_1+\cdots+m_{l-1}+l-1}\tau(l)
\end{align*}
 and by induction obtain $k=l$, $\mu(i)=\tau(i)$, $1\leq i\leq k$.

 \q So we may assume that no $\lambda^h$, with $h>0$, is a beginning term. This implies that all $\mu(i),\tau(j)$, with $i>0$, $j>0$, are restricted $2$-special partitions.
 
\q If $k=0$ then $\lambda=\mu(0)$.  If $m_0=0$ then $\lambda=\lambda^0$ and the result is clear. If $m_0>0$ then $\lambda^N$ is an end term. Hence $\lambda^N$ is an end term of some $\tau(j)$ and  since there is only one beginning term,  namely $\lambda^0$, we must have $j=0$ and $\lambda=\tau(0)$, $l=0$.  Thus we may assume $k>0$ and, for the same reason, that $l>0$.

 \q Thus we have that $\mu(i)$ is $2$-special for $0<i\leq k$ and $\tau(j)$ is $2$-special for $0<j\leq l$.  We thus have  $\mu(k)=\lambda^N=\tau(l)$ and 
 \begin{align*}\mu(0)&+p\mu(1)^{m_0+1}+\cdots+p^{m_0+\cdots+m_{k-2}+(k-1)}\mu(k-1)\cr
 &=\tau(0)+p^{n_0+1}\tau(1)+\cdots+p^{n_0+\cdots+n_{l-2}+(l-1)}\tau(l-1).
 \end{align*}
Again we obtain inductively that $k-1=l-1$, so $k=l$, and $\mu(i)=\tau(i)$, for all $0\leq i\leq k$.

\end{proof}

\q An irreducible module will be called primitive of index $m$ (resp. standard) if  it has the form $L(\lambda)$, where $\lambda$ is primitive of index $m$ (resp. standard). 
We now come to the main result of the paper, which gives a precise description of the irreducible modules that occur as a composition factor of a tensor product of symmetric powers of the natural module.

\begin{theorem} A partition is $2$-good if and only if it is standard.

\end{theorem}

\begin{proof}  Here good will mean $2$-good and special will mean $2$-special.

\q We first show that any standard partition  is good. Suppose $\lambda$ is a standard partition and write $\lambda=\alpha+p^{m+1}\mu$, with $\alpha$ primitive of index $m$ and $\mu$ standard. We may assume inductively that $\mu$ is good. If $m=0$ we get immediately that $\lambda=\alpha+p\mu$ is  good, from \cite{DG4}, Corollary 2.9(ii).  Suppose now  $m>0$ and write $\alpha=\alpha^0+p\alpha^1+\ldots+p^m\alpha^m$ with $\alpha^0$ a beginning term, $\alpha^1,\ldots,\alpha^{m-1}$ middle terms and $\alpha^m$ an end term.  We write $\alpha^i=\beta^i+\omega_{r_i}$,  with $\beta^i$ a beginning term, $r_i\geq 1$, for $1\leq i\leq m-1$. We write $\alpha^m=\beta^m+\omega_{r_m}$, with $\beta^m$ restricted and  special and $r_m\geq 1$.   Now $\alpha^0+p\omega_{r_1},  \beta^1+p\omega_{r_2}, \ldots, \beta^{m-1}+p\omega_{r_m}$  are special, by \cite{DG4}, Theorem 6.5, and 
$\beta^m$ is   special so that  $\mu$ is good. Thus by  \cite{DG4}, Corollary 2.9(ii),  every composition factor of
\begin{align*}L(\alpha^0+p\omega_{r_1})&\otimes L(\beta^1+p\omega_{r_2})^F\otimes \cdots\\
 &\otimes L(\beta^{m-1}+p\omega_{r_m})^{F^{m-1}} 
\otimes L(\beta^m)^{F^m}\otimes L(\mu)^{F^{m+1}}
\end{align*}
is  good.  But this module has highest weight $\lambda$ and so $\lambda$ is  good.

\q We now show the converse.   An irreducible good  module is a composition factor of 
$$L(\mu(0))\otimes L(\mu(1))^F\otimes \cdots \otimes L(\mu(h))^{F^h}$$
for some $h\geq 0$ and   special partitions $\mu(0),\ldots,\mu(h)$, by \cite{DG4}, Corollary 2.9(ii).    We consider the set ${\mathcal S}$ of all sequences $\mu=(\mu(0),\mu(1),\ldots)$ of special partitions,   with $\mu(j)=0$ for $j\gg 0$.  For $\mu=(\mu(0),\mu(1),\ldots )\in {\mathcal S}$ we define  the module
$$V(\mu)=L(\mu(0))\otimes L(\mu(1))^F\otimes \cdots$$

and the partition 
$$f(\mu)=\mu(0)+p\mu(1)+\cdots.$$
The result will be proved if we establish the claim that every composition factor of $V(\mu)$, for $\mu\in {\mathcal S}$,  is standard.  Assume, for a contradiction,that the claim  is false and that $\mu\in {\mathcal S}$ is such that $f(\mu)$ has minimal degree subject to the condition that $V(\mu)$ has a non-standard composition factor.  We denote by $N$ a non-standard composition factor.

\q  Note that for $\nu=(0,0,\ldots)$ we have $V(\nu)=L(0)$, which is standard, and so $\mu\neq (0,0,\ldots)$. 

\q  Let $\lambda=\mu(0)$. Suppose that  $\lambda$ is restricted. Let $\tau= (\mu(1),\mu(2),\ldots)\in {\mathcal S}$. Then we have $V(\mu)=L(\lambda)\otimes V(\tau)^F$. Then $N$ is a composition factor of $L(\lambda)\otimes M^F$, for some composition factor  $M$ of $V(\tau)$. By the minimality of $\mu$, we have $M=L(\sigma)$, for standard partition $\sigma$ and hence, by Steinberg's tensor product theorem 
$N=L(\lambda+  p\sigma)$, which is standard.

\q Thus $\lambda$ is not restricted and, by Proposition 2.2,   we may write $\lambda=\lambda^0+p\omega_r$, with $\lambda^0$ a beginning term and $r\geq 1$. Setting $\theta=\mu(1)$ we now have 
$$V(\mu)=L(\lambda^0)\otimes Z^F\otimes V(\xi)^{F^2}$$
where $Z=L(\omega_r)\otimes L(\theta)$ and $\xi=(\mu(2),\mu(3), \ldots)\in {\mathcal S}$.

\q Hence $N$ is a composition factor of $L(\lambda^0)\otimes Z_1^F\otimes V(\xi)^{F^2}$, for some composition factor $Z_1$ of $Z$. By Theorem 2.10, $Z_1$ has the form $L(\sigma+\omega_s)$, for some special partition $\sigma$ and $s\geq 0$.

\q The argument now divides into three cases.

\bs

{\it Case (i)\,: $\sigma$ is not restricted.}

\bs

\q By \cite{DG4}, Theorem 6.5, we  have $\sigma=\sigma^0+p\omega_t$, for some beginning term $\sigma^0$ and $t\geq  1$. Thus $N$ is a composition factor of $L(\lambda^0)\otimes L(\sigma^0+\omega_s)^F\otimes U^F$, where 
$$U=L(\omega_t)^F\otimes V(\xi)^F=V(\phi)$$
with $\phi=(p\omega_t,\mu(2),\mu(3),\ldots)$.  Hence $N$ is a composition factor of $L(\lambda^0)\otimes L(\sigma^0+\omega_s)^F \otimes P^F$, for some composition factor $P$ of $V(\phi)$.  By the minimality of $\mu$ we have $P=L(\gamma)$, for some standard partition $\gamma$. We may write $\gamma=\alpha+p^{m+1}\beta$, with $\alpha$ primitive of index $m$ and $\beta$ standard. Moreover, since $V(\phi)=(L(\omega_t)\otimes V(\xi))^F$, we have that $\alpha$ is divisible by $p$. 

\q If $m=0$ then $\alpha$ is restricted and hence $0$  so that  $P=L(\beta)^F$ and so $N=L(\lambda^0)\otimes L(\sigma^0+\omega_s)^F\otimes L(\beta)^{F^2}$. If $\sigma^0+\omega_s$ is not special then it is an end term and $L(\lambda^0)\otimes L(\sigma^0+\omega_s)^F$ is primitive of index $1$ and $N=L(\lambda^0)\otimes L(\sigma^0+\omega_s)^F\otimes L(\beta)^{F^2}$ is standard. If $\sigma^0+\omega_s$ is special then both $\lambda^0$ and $\sigma^0+\omega_s$ are restricted and special, $\beta$ is standard and so $N=L(\lambda^0)\otimes L(\sigma^0+\omega_s)^F\otimes L(\beta)^{F^2}$ is standard.

\q Assume now that $m>0$ and write $\alpha=\alpha^0+p\alpha^1+\cdots+p^m\alpha^m$, with $\alpha^0$ a beginning term, $\alpha^1,\ldots,\alpha^{m-1}$ middle terms and $\alpha^m$ an end term. Since $\alpha$ is divisible by $p$, we have $\alpha^0=0$. Hence we have
$$N=L(\lambda^0)\otimes L(\sigma^0+\omega_s)^F \otimes L(\alpha^1)^{F^2}\otimes \cdots\otimes L(\alpha^m)^{F^{m+1}}\otimes L(\beta)^{F^{m+2}}.$$

\q If $\sigma^0+\omega_s$ is a beginning term then $(\sigma^0+\omega_s)+p\alpha^1+\cdots+ p^m\alpha^m$ is primitive of index $m$ and so $L(\sigma^0+\omega_s)\otimes L(\alpha^1)^F\otimes \cdots\otimes L(\alpha^m)^{F^m}\otimes L(\beta)^{F^{m+1}}$ is standard. Hence 
$$N=L(\lambda^0)\otimes (L(\sigma^0+\omega_s) \otimes L(\alpha^1)^{F}\otimes \cdots\otimes L(\alpha^m)^{F^{m}}\otimes L(\beta)^{F^{m+1}})^F$$
is standard.

\q If $\sigma^0+\omega_s$ is not a beginning term then it is a middle term so $\lambda^0$ is a beginning term, $\sigma^0+\omega_s,\alpha^1,\ldots,\alpha^{m-1}$ are middle terms and $\alpha^m$ is an end term so that 
$$\delta=\lambda^0+p(\sigma^0+\omega_s)+p^2\alpha^1+\cdots+p^{m+1}\alpha^m$$
is primitive of index $m+1$ and hence $\delta+p^{m+2}\beta$ is standard, i.e., $N$ is standard.

\bs

{\it Case (ii)\,: $\sigma$ is restricted and $\sigma+\omega_s$  is restricted.}

\bs

\q Now $N$ is a composition factor of $L(\lambda^0)\otimes L(\sigma+\omega_s)^F\otimes V(\xi)^{F^2}$ and so $N=L(\lambda^0)\otimes L(\sigma+\omega_s)^F\otimes Q^{F^2}$ for some composition factor $Q$ of $V(\xi)$. By minimality of $\mu$ we have $Q=L(\gamma)$ for some standard partition $\gamma$. 

\q If $\sigma+\omega_s$ is special then $\lambda^0$ and $\sigma+\omega_s$ are restricted special so that $\lambda^0+p(\sigma+\omega_s)+p^2\gamma$ is standard, i.e., $N$ is standard. 

\q If  $\sigma+\omega_s$  is not special then it is an end term and $\lambda^0+p(\sigma+\omega_s)$ is primitive of index $1$ and so $\lambda^0+p(\sigma+\omega_s)+p^2\gamma$ is standard, i.e., $N$ is standard.

\bs

{\it Case (iii)\,: $\sigma$ is restricted but  $\sigma+\omega_s$  is not.}

\bs

\q Thus we have $\sigma_1\geq p-1$ and so, by \cite{DG4}, Theorem 6.5, $\sigma=(p-1)^uc+(p-1)^vd$, for some $u,v\geq 0$, $p-1>c\geq 0$, $p-1>d\geq 0$. It is not difficult to check that, since $\sigma$ is restricted and $\sigma+\omega_s$ is not, we can write $\sigma=(p-1)^ka+(p-1)^s$ and $\sigma+\omega_s=(p-1)^ka+p\omega_s$, with $k\geq 0$, $p-1>a\geq 0$. Hence $N$ is a composition factor of 
$$L(\lambda^0)\otimes L(\zeta)^F \otimes (L(\omega_s)\otimes V(\xi))^{F^2}$$
where $\zeta=(p-1)^ka$.  Thus $N$ has the form $L(\lambda^0)\otimes L(\zeta)^F\otimes L^{F^2}$, where $L$ is a composition factor of $L(\omega_s)\otimes V(\xi)$.  Thus $L^F$ is a composition factor of $L(p\omega_s)\otimes V(\xi)^F$ and, by minimality of $\mu$, we must have $L^F=L(\delta)$, for some standard $\delta$. We write $\delta=\alpha+p^{m+1}\beta$, where $\alpha$ is primitive of index $m$ and $\beta$ is standard. 

\q If $m=0$ then, since $L(\delta)=L^F$, we have $\alpha=0$. Thus we have $N=L(\lambda^0)\otimes L((p-1)^ka)^F\otimes L(\beta)^{F^2}$, i.e., $N=L(\eta)$, where $\eta=\lambda^0+p\zeta+p^2\beta$. Now $\lambda^0$ and $\zeta$ are  restricted special and $\beta$ is standard so that $\eta$ is standard, i.e., $N$ is standard.

\q So we may suppose $m>0$. We write $\alpha=\alpha^0+p\alpha^1+\cdots+p^m\alpha^m$, with $\alpha^0$ a beginning term, $\alpha^1,\ldots,\alpha^{m-1}$ middle terms and $\alpha^m$ an end term. Since $L(\alpha)\otimes L(\beta)^{F^{m+1}}$ is $L^F$ we must have $\alpha^0=0$. Hence we have
$$N=L(\lambda^0)\otimes L(\zeta)^F\otimes L(\alpha^1)^{F^2}\otimes \cdots \otimes L(\alpha^m)^{F^{m+1}}\otimes L(\beta)^{F^{m+2}}.$$

\q Recall that $\zeta=(p-1)^ka$. If $k=0$ then $\zeta$ is a beginning term. Also, $\alpha^1,\ldots,\alpha^{m-1}$ are middle terms and $\alpha^m$ is an end term. Hence $\epsilon=\zeta+p\alpha^1+\cdots+p^m\alpha^m$ is primitive of index $m$ and therefore $\epsilon+p^{m+1}\beta$ is standard. Hence $\nu=\lambda^0+p(\epsilon+p^{m+1}\beta)$ is standard, i.e., $N=L(\nu)$ is standard.

\q If $k>0$ then $\zeta=(p-1)^ka$ is a middle term. Thus $\lambda^0$ is a beginning term, $\zeta,\alpha^1,\ldots,\alpha^{m-1}$ are middle terms and $\alpha^m$ is an end term so that $\gamma=\lambda^0+p\zeta+p^2\alpha^1+\cdots+p^{m+1}\alpha^m$ is primitive of index $m+1$ and $\gamma+p^{m+2}\beta$ is standard, i.e., $N$ is standard.

\bs

Expressing this in terms of modules and bearing in mind the uniqueness statement,  Lemma 3.5, we get the following result.

\bs

\begin{theorem}  A composition factor $L$ of the tensor product  $S(E)\otimes S(E)$, of two copies of the symmetric algebra of the natural ${\rm GL}_n(K)$-module $E$  has the form 
$$L=L(\lambda(0))\otimes L(\lambda(1))^{F^{m_0+1}}\otimes \cdots \otimes L(\lambda(k))^{F^{m_0+\cdots+m_{k-1}+k}}$$
for some $k\geq 0$ and primitive partitions $\lambda(0),\ldots,\lambda(k)\in \Lambda^+(n)$, where $m_i$ is the index of $\lambda(i)$, for $0\leq i<k$.  Moreover, if $L$ is non-trivial there is a unique such expression, with $\lambda(k)\neq 0$.  Further every irreducible module of the above form occurs as a  composition factor.

\end{theorem}

\end{proof}

{\bf Remark 3.8}     \, Recall from \cite{DG2}, that the divisibility index $\divind(V)$, of a non-zero polynomial ${\rm GL}_n(K)$-module $V$ is the largest integer $k$ such that $V\cong V'\otimes D^{\otimes k}$, for some polynomial module $V'$. Here $D$ denotes the one dimensional module afforded by the determinant representation. We call $V$ critical if $\divind(V)=0$. 
If $V$ is not critical then $\divind(V)=\divind(V\otimes D^*)+1$.

\q For $\lambda\in \Lambda^+(n)$ we write $\divind(\lambda)$ for the divisibility index of the injective envelope $I(\lambda)$ of $L(\lambda)$ in the category of polynomial modules. We call $\lambda$ critical if $I(\lambda)$ is critical. 

\q We now take $n=3$. (The case $n=2$ was considered in \cite{DG2}, Section 5.)  Then, by \cite{DG2}, Lemma 3.9, $\lambda$ is critical if and only if $L(\lambda)$ is a composition factor of $S(E)\otimes S(E)$, i.e., if and only if $\lambda$ is $2$-good.  Thus, from Theorem 3.7, we have an explicit description of the critical partitions $\lambda\in \Lambda^+(3)$, i.e., $\lambda$ is critical if and only if for some $k\geq 0$ we have
$$\lambda=\lambda(0)+p^{m_0+1}\lambda(2)+\cdots+p^{m_0+\cdots+m_{k-1}+k}\lambda(k)$$
for primitive $\lambda(i)\in \Lambda^+(3)$, $0\leq i\leq k$, where $\lambda(i)$ has index $m_i$, for $0\leq i<k$.

\q Thus, by \cite{DG2}, Theorem 4.1, we may describe all $\lambda\in \Lambda^+(3)$ such that the restriction of $I(\lambda)$ to the first infinitesimal subgroup $G_1$ of ${\rm GL}_3(K)$ is injective. This amounts to the following:

\bs

\q $I(\lambda)$ is injective as a $G_1$-module if and only if either:

(i) $\lambda^0_1\geq 2p-2$; or 

(ii) $p-2\leq \lambda^0_1<2p-2$ and $\barlambda$ is not critical; or 

(iii) $\lambda^0_1<p-2$,  $\barlambda$ and $\barlambda-\omega_3$ are not critical.

\bs

\q Here we are writing $\lambda=\lambda^0+p\barlambda$, with $\lambda^0$ restricted and $\barlambda\in \Lambda^+(3)$.

\section*{Acknowledgement}

The second author gratefully acknowledges the financial  support of EPSRC Grant EP/L005328/1.

\bs\bs\bs





\end{document}